\documentclass[11pt]{amsart}
\usepackage{extarrows}
\usepackage[colorlinks, citecolor=blue, dvipdfm, pagebackref]{hyperref}
\usepackage[all]{xy}

\newtheorem{theorem}{Theorem}[section]
\newtheorem{lemma}[theorem]{Lemma}

\theoremstyle{definition}
\newtheorem{definition}[theorem]{Definition}
\newtheorem{question}[theorem]{Question}

\newtheorem{corollary}[theorem]{Corollary}
\newtheorem{remark}[theorem]{Remark}

\theoremstyle{remark}

\newcommand{\be}{\begin{equation}}
\newcommand{\ee}{\end{equation}}

\numberwithin{equation}{section}

\begin{document}

\title{Vanishing theorems and rational connectedness on holomorphic tensor fields}

\author{Ping Li}
\address{School of Mathematical Sciences, Fudan University, Shanghai 200433, China}

\email{pinglimath@fudan.edu.cn;~pinglimath@gmail.com}
\thanks{The author was partially supported by the National
Natural Science Foundation of China (Grant No. 12371066).}




\begin{abstract}
A vanishing theorem for uniformly RC $k$-positive Hermitian holomorphic vector bundles is established. It turns out that the holomorphic tangent bundle of a compact complex manifold equipped with a positive $k$-Ricci curvature K\"{a}hler metric is uniformly RC $k$-positive. Two main applications are presented. The first one is to deduce that spaces of some holomorphic tensor fields on such K\"{a}hler or more generally K\"{a}hler-like Hermitian manifolds are trivial, generalizing some recent results.
The second one is to show that a compact K\"{a}hler manifold whose holomorphic tangent bundle can be endowed with either a uniformly RC $k$-positive Hermitian metric or a positive $k$-Ricci curvature K\"{a}hler-like Hermitian metric is projective and rationally connected.
\end{abstract}

\maketitle


\section{Introduction}\label{introduction and main results}
A central topic in differential geometry is how the curvature conditions restrict the topology. In the case of Ricci curvature (abbreviated by ``Ric") in K\"{a}hler geometry this principle goes back to Bochner, who proved that (\cite{Bo1}, \cite{Bo2}, \cite{YB}) the condition of $\text{Ric}>0$ or $\text{Ric}<0$ on compact K\"{a}hler manifolds or the existence of K\"{a}hler-Einstein metrics imposes heavy restrictions on
their holomorphic tensor fields. These conditions can now be reformulated in terms of the first Chern class, thanks to the celebrated Calabi-Yau theorem and the Aubin-Yau theorem (\cite{Yau77}). In order to precisely state Bochner's results, let (throughout this article) $TM$ and $T^{\ast}M$ be respectively the (\emph{holomorphic}) tangent and cotangent bundle of a compact complex manifold $M$, and
\be\label{holomorphic tensor field}\Gamma^p_q(M):=H^0\big(M,(TM)^{\otimes p}\otimes(T^{\ast}M)^{\otimes q}\big),\qquad (p, q\in\mathbb{Z}_{\geq 0})\ee
the space of \emph{$(p,q)$-type holomorphic tensor fields} on $M$. Bochner's results can be stated as follows (\cite{Ko80-1}, \cite{Ko80-2}, \cite[p. 57]{KH83}).
\begin{theorem}[Bochner, Calabi-Yau, Aubin-Yau]\label{Bochner}
Let $M$ be an $n$-dimensional compact K\"{a}hler manifold.
\begin{enumerate}
\item
If $c_1(M)$ is quasi-positive, then there exists a positive constant $C=C(M)$ such that $\Gamma^p_q(M)=0$ when $q>C\cdot p$, and consequently $\Gamma^0_q(M)=0$ when $q\geq 1$. In particular the Hodge numbers $h^{q,0}(M)=0$ when $1\leq q\leq n$.

\item
If $c_1(M)$ is quasi-negative, then there exists a positive constant $C=C(M)$ such that $\Gamma^p_q(M)=0$ when $p>C\cdot q$, and consequently $\Gamma^p_0(M)=0$ when $p\geq 1$.

\item
If $c_1(M)<0$, then $\Gamma^p_q(M)=0$ when $p>q$.
\end{enumerate}
\end{theorem}
\begin{remark}
\begin{enumerate}
\item
The condition quasi-positivity (resp. quasi-negativity) of $c_1(M)$ means that there exists a closed $(1,1)$-form representing $c_1(M)$ which is nonnegative (resp. non-positive) everywhere and positive (resp. negative) somewhere.

\item
It is well-known that a compact K\"{a}hler manifold with Hodge number $h^{2,0}=0$ is projective (\cite[p. 143]{MK71}). So a direct consequence of Theorem \ref{Bochner} is that a compact K\"{a}hler manifold with quasi-positive $c_1$ must be projective.
\end{enumerate}
\end{remark}
It is well-known that the condition of $\text{Ric}>0$, which is equivalent to be Fano due to the Calabi-Yau theorem, implies simple-connectedness (\cite{Ko61}). This was further strengthened to be rationally connected by Campana (\cite{Ca}) and Koll\'{a}r-Miyaoka-Mori (\cite{KMM}) independently, which was conjectured in Yau's influential problem list (\cite[Problem 47]{Yau82}). Recall that a complex manifold is called \emph{rationally connected} if any two points on it can be joined by a rational curve. Rational connectedness is an important tool/notion in algebraic geometry and for projective manifolds it implies simple-connectedness (\cite[Coro. 4.18]{Deb}).

In view of the subtle relationship between Ric and the holomorphic sectional curvature $H$ (\cite[p. 181]{Zh}), one may expect that the above-mentioned conclusions still hold when the condition on Ric is replaced by that on $H$. The simple-connectedness under the condition $H>0$ is well-known (\cite{Ts}).
A recent important result due to Wu-Yau, Tosatti-Yang and Diverio-Trapani (\cite{WY16-1}, \cite{TY17}, \cite{DT19}, \cite{WY16-2}) implies that the quasi-negativity of $H$ implies $c_1(M)<0$, and hence the conclusions of parts (2) and (3) in Theorem \ref{Bochner} remain true when $H$ is quasi-negative. Yau also conjectured in \cite[Problem 47]{Yau82} that a compact K\"{a}hler manifold with $H>0$ is projective and rationally connected. Assuming the projectivity this was proved by Heier-Wong (\cite{HW}). Shortly afterwards the projectivity was proved by Yang in \cite{Yang18} by showing that the Hodge number $h^{2,0}=0$, who also provided an alternative proof of the rational connectedness therein. Recently Yang introduced in \cite{Yang20} the notion of \emph{uniform RC-positivity} on Hermitian holomorphic vector bundles, which is satisfied by the tangent bundle of a compact K\"{a}hler manifold with $H>0$, and deduced that a compact K\"{a}hler manifold whose tangent bundle is uniformly RC-positive with respect to a (possibly different) Hermitian metric is projective and rationally connected (\cite[Thm 1.3]{Yang20}).

In contrast to the negative case, in general the condition $H>0$ is \emph{not} able to yield $c_1(M)>0$, as exhibited by Hitchin (\cite{Hi75}) that the Hirzebruch surfaces \be\label{Hitchin example}F:=\mathbb{P}\big(\mathcal{O}_{\mathbb{P}^1}(-k)
\oplus\mathcal{O}_{\mathbb{P}^1}\big)\qquad(k\geq 2)\ee
admit K\"{a}hler metrics with $H>0$ but $c_1(F)$ are \emph{not} positive (cf. \cite[p. 949]{Yang16}). More related examples can be found in \cite{AHZ} and \cite{NZ0}. Even so, we can still ask whether or not the conclusions of part (1) in Theorem \ref{Bochner} hold when $H>0$. In fact, Yang showed that (\cite[Thm 1.7]{Yang18}) the condition $H>0$ leads to the Hodge numbers $h^{q,0}=0$ for $1\leq q\leq n$ and thus provides positive evidence towards this validity. Recently the author showed that the conclusions of part (1) in Theorem \ref{Bochner} are true when $H>0$ (\cite{Li21}). Indeed what we proved in \cite[Thm 1.5]{Li21} is in a more general setting, i.e. for those Hermitian metrics whose Chern curvature tensors behave like the usual K\"{a}hler curvature tensors.

The \emph{$k$-Ricci curvatures} $\text{Ric}_k$ for $1\leq k\leq n$ were introduced by Ni (\cite{Ni21-1}) in his study of the $k$-hyperbolicity of a compact K\"{a}hler manifold. $\text{Ric}_k$ coincides with $H$ when $k=1$ and with Ric when $k=n$ and so they interpolate between $H$ and $\text{Ric}$. Ni showed that the condition $\text{Ric}_k>0$ for some $k\in\{1,\ldots,n\}$ yields rational connectedness and $h^{q,0}=0$ for $1\leq q\leq n$, and hence the simple-connectedness and projectivity (\cite[Thms 1.1, 4.2]{Ni21-2}). Chu-Lee-Tam showed that (\cite[Thm 5.1]{CLT}) the condition $\text{Ric}_k<0$ for some $k\in\{1,\ldots,n\}$ also implies $c_1(M)<0$ (see \cite{LNZ} by Li-Ni-Zhu for an alternative proof). Hence the conclusions of parts (2) and (3) in Theorem \ref{Bochner} hold true under the condition $\text{Ric}_k<0$ for some $k\in\{1,\ldots,n\}$.

In view of the discussions above, we may wonder whether the conclusions of part (1) in Theorem \ref{Bochner} are true when $\text{Ric}_k>0$ for some $k\in\{1,\ldots,n\}$. One purpose in this article is, as an application of the main results established, to affirmatively prove it. Another major purpose is to introduce the notion of \emph{uniform RC $k$-positivity}, which coincides with uniform RC-positivity when $k=1$, and extend Yang's differential-geometric criterion for rational connectedness from the case of uniform RC-positivity to that of uniform RC $k$-positivity for arbitrary $k$.

\section{Main results}\label{main results}
Our main results and applications are stated in this section. The following concept is inspired by that of uniform RC-positivity introduced by Yang (\cite{Yang20}) as well as BC $k$-positivity by Ni (\cite{Ni21-2}).
\begin{definition}\label{uniform rc k positivity}
A Hermitian holomorphic vector bundle $(E,h)$ over an $n$-dimensional Hermitian manifold $(M,\omega)$ is called \emph{uniformly RC $k$-positive} (resp. \emph{uniformly RC $k$-negative}) $(1\leq k\leq n)$ at $x\in M$ if there exists a $k$-dimensional subspace $\Sigma\subset T_xM$ such that for every nonzero vector $u\in E_x$ (the fiber of $E$ at $x$),
\be\label{RC-positivity}
R^{(E,h)}(\Sigma;u,\overline{u}):=
\sum_{i=1}^kR^{(E,h)}(E_i,\overline{E_i},u,\overline{u})>0~(\text{resp. $<0$}),\ee
where $R^{(E,h)}$ is the Chern curvature tensor of $(E,h)$ and $\{E_1,\ldots,E_k\}$ a unitary basis of $\Sigma$ with respect to $\omega$. If this holds for any $x\in M$, then it is called \emph{uniformly RC $k$-positive} (resp. \emph{uniformly RC $k$-negative}).
\end{definition}
\begin{remark}
\begin{enumerate}
\item
The uniform RC-positivity in \cite{Yang20} is exactly the uniform RC $1$-positivity in our notion, whose definition is indeed irrelevant to the metric $\omega$. Nevertheless, the definition of uniform RC $k$-positivity for $k\geq2$ relies on the metric $\omega$. A different but closely related notion called \emph{BC $k$-positivity} was introduced by Ni in \cite[p. 282]{Ni21-2}. We refer to Definition \ref{uniform RC kl definition} and Remark \ref{uniform RC remark} for more details on various positivity concepts.

\item
It is obvious that $(E,h)$ is uniformly RC $k$-positive if and only if the dual bundle $(E^{\ast},h^{\ast})$ is uniformly RC $k$-negative, where $h^{\ast}$ is the metric induced from $h$.
\end{enumerate}
\end{remark}
The following vanishing theorem for uniformly RC $k$-positive vector bundles is a key technical tool.

\begin{theorem}\label{RC vanishing theorem}
Let $(E,h)$ be a uniformly RC $k$-positive holomorphic vector bundle over an $n$-dimensional compact Hermitian manifold $(M,\omega)$ for some $k\in\{1,\ldots,n\}$, and $F$ any holomorphic vector bundle over $M$. Then there exist positive constants $C_1=C_1(h,\omega)$ and $C_2=C_2(F,h,\omega)$ such that
\be\label{RC vanishing formula}H^0\big(M,E^{\otimes p}\otimes (E^{\ast})^{\otimes q}\otimes F^{\otimes m}\big)=0\ee
for all $p,q,m\in\mathbb{Z}_{\geq 0}$ with $q> C_1\cdot p+C_2\cdot m$.
\end{theorem}

By taking $E=TM$ or $T^{\ast}M$ and $F$ trivial in Theorem \ref{RC vanishing theorem}, we have
\begin{corollary}\label{corollary}
Let (M, $\omega$) be an $n$-dimensional compact Hermitian manifold such that $(TM,\omega)$ is uniformly RC $k$-positive (resp. $k$-negative) over $(M,\omega)$ for some $k\in\{1,\ldots,n\}$. Then there exists a positive constant $C=C(\omega)$ such that $\Gamma^p_q(M)=0$ when $q>C\cdot p$ (resp. $p>C\cdot q$), and consequently $\Gamma^0_q(M)=0$ (resp. $\Gamma^p_0(M)=0$) when $q\geq 1$ (resp. $p\geq 1$). In particular in the former case the Hodge numbers $h^{q,0}(M)=0$ when $1\leq q\leq n$.
\end{corollary}

The main motivation to introduce the concept of uniform RC $k$-positivity is that the tangent bundle of a compact K\"{a}her manifold $(M,\omega)$, or more generally a compact K\"{a}hler-like Hermitian manifold $(M,\omega)$, with the condition $\text{Ric}_k(\omega)>0$ (resp. $\text{Ric}_k(\omega)<0$) turns out to be uniformly RC $k$-positive (resp. negative). To this end, let us recall the following notion, which was first proposed and investigated in detail by B. Yang and Zheng in \cite{YZ}.
\begin{definition}
Let $(M,\omega)$ be a Hermitian manifold and $R$ the \emph{Chern curvature tensor} of $\omega$. The Hermitian metric $\omega$ is called \emph{Chern-K\"{a}hler-like} (abbreviated by CKL) if \be\label{CKL}
R(X,\overline{Y},Z,\overline{W})=R(Z,\overline{Y},X,\overline{W})\ee for any $(1,0)$-type tangent vectors $X$, $Y$, $Z$ and $W$.
\end{definition}
\begin{remark}\label{remark}
\begin{enumerate}
\item
When $\omega$ is K\"{a}hler, $R$ is the usual K\"{a}hler curvature tensor and (\ref{CKL}) is satisfied.
By taking the complex conjugation, $(\ref{CKL})$ implies that
$$R(X,\overline{Y},Z,\overline{W})=
\overline{R(Y,\overline{X},W,\overline{Z})}=
\overline{R(W,\overline{X},Y,\overline{Z})}=
R(X,\overline{W},Z,\overline{Y})
.$$
Therefore the condition (\ref{CKL}) ensures that $R$ obeys \emph{all} the symmetries satisfied by a K\"{a}hler metric and thus the term CKL is justified.

\item
There are many \emph{non-K\"{a}hler} Hermitian metrics which are CKL (\cite[p. 1197]{YZ}). On the other hand, Yang-Zheng showed that a CKL Hermitian metric $\omega$ must be balanced, i.e., $d\omega^{n-1}=0$ (\cite[Thm 3]{YZ}). Hence CKL Hermitian metrics interpolate between K\"{a}hler and balanced metrics.
\end{enumerate}
\end{remark}
With this notion understood, our second main result is the following
\begin{theorem}\label{rick lead to uniform}
Let $(M,\omega)$ be an $n$-dimensional compact Chern-K\"{a}hler-like Hermitian manifold and $\text{Ric}_{k}(\omega)>0$ (resp. $\text{Ric}_{k}(\omega)<0$) for some $k\in\{1,\ldots,n\}$. Then $(TM,\omega)$ is uniformly RC $k$-positive (resp. uniformly RC $k$-negative) over $(M,\omega)$.
\end{theorem}
Combining Corollary \ref{corollary} with Theorem \ref{rick lead to uniform}, it yields the following desired vanishing theorem, which extends \cite[Thm 1.5]{Li21} from the case $\text{Ric}_1$ to arbitrary $\text{Ric}_k$.
\begin{theorem}\label{CKL result}
Let $(M,\omega)$ be an $n$-dimensional compact Chern-K\"{a}hler-like Hermitian manifold.
\begin{enumerate}
\item
If $\text{Ric}_k(\omega)>0$ for some $k\in\{1,\ldots,n\}$,  then there exists a positive constant $C=C(\omega)$ such that $\Gamma^p_q(M)=0$ when $q>C\cdot p$, and consequently $\Gamma^0_q(M)=0$ when $q\geq 1$. In particular the Hodge numbers $h^{q,0}(M)=0$ when $1\leq q\leq n$.

\item
If $\text{Ric}_k(\omega)<0$ for some $k\in\{1,\ldots,n\}$, then there exists a positive constant $C=C(\omega)$ such that $\Gamma^p_q(M)=0$ when $p>C\cdot q$, and consequently $\Gamma^p_0(M)=0$ when $p\geq 1$.
\end{enumerate}
\end{theorem}
\begin{remark}
When $\omega$ is K\"{a}hler, part $(2)$ in Theorem \ref{CKL result} follows from Theorem \ref{Bochner} and the results due to Wu-Yau et al. and Chu-Lee-Tam, as mentioned above. Nevertheless, if the CKL metric $\omega$ is \emph{non-K\"{a}hler}, the results in part $(2)$ are still new.
\end{remark}

Our second major application is a differential-geometric criterion for the rational connectedness of compact K\"{a}hler manifolds. The following nice criterion for rational connectedness was established in \cite[Thm 1.1]{CDP}, on which both \cite[Thm 1.3]{Yang20} and \cite[Thm 1.1]{Ni21-2} are based.
\begin{theorem}[Campana-Demailly-Peternell]\label{CDP}
Let $M$ be a projective manifold. Then $M$ is rationally connected if and only if for any ample line bundle $L$ on $M$, there exists a positive constant $C=C(L)$ such that
$$H^0\big(M,(T^{\ast}M)^{\otimes q}\otimes L^{\otimes m}\big)=0$$
for $q,m\in\mathbb{Z}_{\geq 0}$ with $q> C\cdot m$.
\end{theorem}

We now have the following criterion for projectivity and rational connectedness.
\begin{theorem}\label{rational connectedness}
Let $M$ be an $n$-dimensional compact K\"{a}hler manifold. Then it is projective and rationally connected (and hence simply-connected) provided one of the following two conditions holds true.
\begin{enumerate}
\item
There exists a (possibly different) Hermitian metric $\omega$ such that
$(TM,\omega)$ is uniformly RC $k$-positive over $(M,\omega)$ for some $k\in\{1,\ldots,n\}$.

\item
$M$ has a Chern-K\"{a}hler-like Hermitian metric $\omega$ with $\text{Ric}_k(\omega)>0$ for some $k\in\{1,\ldots,n\}$.
\end{enumerate}
\end{theorem}
\begin{proof}
By Corollary \ref{corollary} and Theorem \ref{CKL result}, either condition implies that the Hodge number $h^{2,0}=0$ and hence the manifold is projective. Since either condition also implies that $(TM,\omega)$ is uniformly RC $k$-positive over $(M,\omega)$ and hence the condition in Theorem \ref{CDP} is satisfied by taking $E=TM$, $p=0$ and $F=L$ in (\ref{RC vanishing formula}).
\end{proof}
\begin{remark}
\begin{enumerate}
\item
Part $(1)$ in Theorem \ref{rational connectedness} extends \cite[Thm 1.3]{Yang20} from the case of $k=1$ to arbitrary $k$. Part $(2)$ technically improves on \cite[Thm 1.1]{Ni21-2} from K\"{a}hler metrics to Chern-K\"{a}hler-like Hermitian metrics.

\item
The Hodge number $h^{2,0}=0$ and hence projectivity can be derived via a somewhat weaker \emph{BC $2$-positivity} by Ni (\cite[p. 280-281]{Ni21-2} (see Definition \ref{uniform RC kl definition} and Remark \ref{uniform RC remark} for more details), and the proof in \cite[Thm 4.6]{Ni21-2} can be adopted to give an alternative one of \cite[Thm 1.3]{Yang20} (cf. \cite[p. 285]{Ni21-2}).
\end{enumerate}
\end{remark}

The rest of this article is organized as follows. Some necessary background materials are collected in Section \ref{Preliminaries}. Sections \ref{proof of rc vanishing} and \ref{proof of theorem rick yield rc} are devoted to the proofs of Theorems \ref{RC vanishing theorem} and \ref{rick lead to uniform} respectively. In Section \ref{further remarks} some related questions and remarks shall be discussed, in which various positivity notions are proposed and their relations are briefly discussed for possible further study in the future.

\section{Background materials}\label{Preliminaries}
We collect in this section some basic facts on Hermitian holomorphic vector bundles and Hermitian and K\"{a}hler manifolds in the form we shall use to prove our main results. A thorough treatment can be found in \cite{Ko87} and \cite{Zh}.

\subsection{Hermitian holomorphic vector bundles}
Let $(E,h)\rightarrow M$ be a Hermitian holomorphic vector bundle of rank $r$ over an $n$-dimensional complex manifold $M$ with canonical Chern connection $\nabla$. The Chern curvature tensor $$R=R^{(E,h)}:=\nabla^2\in\Gamma(\Lambda^{1,1}M\otimes E^{\ast}\otimes E)$$
Here and throughout this article $\Gamma(\cdot)$ is used to denote the space of \emph{smooth} sections for holomorphic vector bundles and the notation $\Gamma^p_q(M)$ in (\ref{holomorphic tensor field}) is reserved throughout this article to denote the space of $(p,q)$-type \emph{holomorphic} tensor fields on $M$.

Take a local frame field $\{s_1,\ldots,s_r\}$ of $E$, whose dual coframe field is denoted by $\{s_1^{\ast},\ldots,s_r^{\ast}\}$, and a local holomorphic coordinates $\{z^1,\ldots,z^n\}$ on $M$. With the Einstein summation convention adopted here and in what follows, the Chern curvature tensor $R$ and the Hermitian metric $h$ can be written as
\begin{eqnarray}\label{curvature tensor}
\left\{ \begin{array}{ll}
\displaystyle R=:\Omega^{\beta}_{\alpha}s^{\ast}_{\alpha}\otimes s_{\beta}=:R^{\beta}_{i\bar{j}\alpha}dz^i\wedge d\bar{z}^j\otimes s^{\ast}_{\alpha}\otimes s_{\beta},\\
~\\
\displaystyle h=(h_{\alpha\bar{\beta}}):=\big(h(s_{\alpha},s_{\beta})\big),\\
~\\
\displaystyle R_{ij\alpha\bar{\beta}}:=R_{ij\alpha}^{\gamma}h_{\gamma\bar{\beta}}.
\end{array} \right.\nonumber
\end{eqnarray}
The Hermitian metric $h(\cdot,\cdot)$ and the induced metrics on various vector bundles arising naturally from $E$ are sometimes denoted by $<\cdot,\cdot>$.

For $u=u^{\alpha}s_{\alpha}\in\Gamma(E)$, $v=v^{\alpha}s_{\alpha}\in\Gamma(E)$, $X=X^i\frac{\partial}{\partial z^i}$, and $Y= Y^i\frac{\partial}{\partial z^i}$, we have
$$R(u)=(\Omega^{\beta}_{\alpha}s^{\ast}_{\alpha}\otimes s_{\beta})(u^{\gamma}s_{\gamma})=
\Omega^{\beta}_{\alpha}u^{\alpha}s_{\beta}
\in\Gamma(\Lambda^{1,1}M\otimes E),$$

$$R_{X\overline{Y}}(u)=\Omega^{\beta}_{\alpha}
(X,\overline{Y})u^{\alpha}s_{\beta}=
R^{\beta}_{i\bar{j}\alpha}X^i\overline{Y^j}u^{\alpha}s_{\beta}
\in\Gamma(E),$$
and therefore,
\be\label{Hermitian formula}
\begin{split}
R(X,\overline{Y},u,\overline{v}):=&
<R_{X\overline{Y}}(u),v>\\
=&
<R^{\beta}_{i\bar{j}\alpha}X^i\overline{Y^j}u^{\alpha}s_{\beta},
v^{\gamma}s_{\gamma}>\\
=&R^{\beta}_{i\bar{j}\alpha}X^i\overline{Y^j}u^{\alpha}\overline{v^{\gamma}}h_{\beta\bar{\gamma}}\\
=&R_{i\bar{j}\alpha\bar{\beta}}X^i\overline{Y^j}u^{\alpha}
\overline{v^{\beta}}.\end{split}
\ee
Here and in what follows we always use capital letters to denote vectors in $TM$ and lowercase letters to denote vectors in $E$, so as to distinguish them.

The formula (\ref{Hermitian formula}) yields an easy but useful fact, which we record in the next lemma for our later reference.
\begin{lemma}\label{diagonal and real lemma}
The linear map $R_{X\overline{X}}(\cdot)$ is a Hermitian transformation: \be\label{Hermitan transformation}<R_{X\overline{X}}(u),v>=<u,R_{X\overline{X}}(v)>,\qquad\text{for any $u,v\in\Gamma(E)$}.\ee
Hence $R_{X\overline{X}}(\cdot)$ is diagonalizable and its eigenvalues are all real at any point in $M$.
\end{lemma}

\subsection{The Ricci and scalar $k$-curvatures}
Let $(M,\omega)$ be an $n$-dimensional Hermitian manifold, $R=R^{(TM,\omega)}$ the Chern curvature tensor.
Denote by $T_xM$ the \emph{$(1,0)$-type} tangent space at $x\in M$ and $\{E_1,\ldots,E_n\}$ a unitary basis of $T_xM$. For $X\in T_xM$, let $H(X)$, $\text{Ric}(X,\overline{X})$ and $S(x)$ be respectively the holomorphic sectional curvature, Chern-Ricci curvature and Chern scalar curvature:
\begin{eqnarray}
\left\{ \begin{array}{ll}
\displaystyle H(X):=\frac{R(X,\overline{X},X,\overline{X})}{|X|^4},\\
~\\
\displaystyle\text{Ric}(X,\overline{X}):=\sum\limits_{i=1}^nR(X,\overline{X},E_i,\overline{E_i}),\\
~\\
\displaystyle S(x):=\sum\limits_{i,j=1}^nR(E_i,\overline{E_i},E_j,\overline{E_j}).
\end{array} \right.\nonumber
\end{eqnarray}
When $\omega$ is K\"{a}hler, $H(X)$, $\text{Ric}(X,\overline{X})$ and $S(x)$ are the usual notions of holomorphic sectional curvature, Ricci curvature and scalar curvature in K\"{a}hler geometry.

For $k\in\{1,\ldots,n\}$, pick a $k$-dimensional subspace $\Sigma\subset T_xM$ and choose a unitary basis $\{E_1,\ldots,E_k\}$ of $\Sigma$. Define
\be\label{def riccik}\text{Ric}_k(x,\Sigma)(X,\overline{X}):=\sum_{i=1}^kR(X,
\overline{X},E_i,\overline{E_i}),\qquad\text{for any $X\in\Sigma$,}\ee
and call it the \emph{$k$-Ricci curvature} of $\omega$, which was introduced by Ni in \cite{Ni21-1}.  The $k$-Ricci curvatures $\text{Ric}_k$ interpolate $H$ and $\text{Ric}$ in the sense that
$$\text{Ric}_1(x,\mathbb{C}X)(X,\overline{X})=|X|^2H(X)\quad \text{and} \quad\text{Ric}_n=\text{Ric}.$$
We say that $\text{Ric}_k(x)>0$ if $\text{Ric}_k(x,\Sigma)(X,\overline{X})>0$ for any $k$-dimensional subspace $\Sigma$ in $T_xM$ and any nonzero $X\in\Sigma$, and $\text{Ric}_k=\text{Ric}_k(\omega)>0$ if $\text{Ric}_k(x)>0$ for any $x\in M$.

Closely related to $\text{Ric}_k$ are
the \emph{$k$-scalar curvatures} $S_k=S_k(\omega)$ ($1\leq k\leq n$) introduced in \cite{NZ} by Ni-Zheng for \emph{K\"{a}hler} metrics, which is the average of $H$ on $k$-dimensional subspaces of $(1,0)$-type tangent spaces. These $\{S_k\}$ interpolate between $H$ ($k=1$) and the usual scalar curvature ($k=n$). They extended Yang's results by showing that the condition $S_k>0$ for a \emph{K\"{a}hler} metric implies that the Hodge numbers $h^{q,0}=0$ for $k\leq q\leq n$ (\cite[Thm 1.3]{NZ}). In particular, the condition $S_2>0$ is enough to gurantee the projectivity.

Here we shall define $S_k=S_k(\omega)$ for $k\in\{1,\ldots,n\}$ in the Hermitian situation, which turns out to be the same as the original one in \cite{NZ} when $\omega$ is K\"{a}hler \big(see (\ref{def integral sk})\big). As above, for $x\in M$, a $k$-dimensional subspace $\Sigma\subset T_xM$ and a unitary basis $\{E_1,\ldots,E_k\}$ of $\Sigma$, define
\be\label{def sk}S_k(x,\Sigma):=\sum_{i,j=1}^k
R(E_i,\overline{E_i},E_j,\overline{E_j}).\ee
It is obvious that in the Hermitian case these $\{S_1,\ldots,S_n\}$ interpolate between $H$ ($k=1$) and the Chern scalar curvature $S(x)$ ($k=n$). Similarly we say that $S_k(x)>0$ if $S_k(x,\Sigma)>0$ for any $k$-dimensional subspace $\Sigma$ in $T_xM$, and $S_k=S_k(\omega)>0$ if $S_k(x)>0$ for any $x\in M$.

\subsection{Integral formulas}
The trick of the proof in the next lemma is usually attributed to Berger, who first applied it to show that for \emph{K\"{a}hler} metrics the sign of $H_x(\cdot)$ determines that of $S(x)$.
\begin{lemma}\label{integral formula lemma}
Let the notation be as above and
$$f(\cdot,\cdot):~T_xM\times T_xM\rightarrow\mathbb{C},\qquad
g(\cdot,\cdot,\cdot,\cdot):~T_xM\times T_xM\times T_xM\times T_xM\rightarrow\mathbb{C}$$
be two smooth maps such that the first variable of $f$ (resp. the first and third variables of $g$) is (resp. are) linear and the second variable of $f$ (resp. the second and fourth variables of $g$) is (resp. are) conjugate-linear.
Then
\be\label{integral formula1}\int_{Y\in\Sigma,~|Y|=1}f(Y,Y)d\theta(Y)=
\frac{\text{V}(\mathbb{S}^{2k-1})}{k}\sum_{i=1}^kf(E_i,E_i),\ee
and
\be\label{integral formula2}\begin{split}&\int_{Y\in\Sigma,~|Y|=1}g(Y,Y,Y,Y)d\theta(Y)\\
=&\frac{\text{V}(\mathbb{S}^{2k-1})}{k(k+1)}\sum_{i,j=1}^k
\Big[g(E_i,E_i,E_j,E_j)+g(E_i,E_j,E_j,E_i)\Big],\end{split}\ee
where $d\theta(Y)$ is the spherical measure on $\{Y\in\Sigma,~|Y|=1\}\cong\mathbb{S}^{2k-1}$ and $\text{V}(\mathbb{S}^{2k-1})$ the volume with respect to it.
\end{lemma}
\begin{proof}
Let $Y=\sum_{i=1}^kY^iE_i\in\Sigma$ and recall the following two classical identities on spherical measure:
\be\label{integral over sphere1}\int_{Y\in\Sigma,~|Y|=1}Y^i\overline{Y^j}d\theta(Y)=
\frac{\delta_{ij}}{k}\cdot\text{V}(\mathbb{S}^{2k-1}),\ee
and
\be\label{integral over sphere2}\int_{Y\in\Sigma,~|Y|=1}Y^i\overline{Y^j}Y^r\overline{Y^s}d\theta(Y)
=\frac{\delta_{ij}\delta_{rs}+\delta_{is}\delta_{rj}}{k(k+1)}
\cdot\text{V}(\mathbb{S}^{2k-1}),
\ee
where $\delta_{ij}$ is the Kronecker delta.

Then
\be\begin{split}\int_{Y\in\Sigma,~|Y|=1}f(Y,Y)d\theta(Y)&=
\int_{Y\in\Sigma,~|Y|=1}f(Y^iE_i,Y^jE_j)d\theta(Y)\\
&=f(E_i,E_j)
\int_{Y\in\Sigma,~|Y|=1}Y^i\overline{Y^j}d\theta(Y)\\
&=\frac{\text{V}(\mathbb{S}^{2k-1})}{k}
\sum_{i=1}^kf(E_i,E_i),\qquad\big(\text{by (\ref{integral over sphere1})}\big)\end{split}\nonumber\ee
and
\be\begin{split}&\int_{Y\in\Sigma,~|Y|=1}
g(Y,Y,Y,Y)d\theta(Y)\\
=&
\int_{Y\in\Sigma,~|Y|=1}
g(Y^iE_i,Y^jE_j,Y^rE_r,Y^sE_s)d\theta(Y)\\
=&g(E_i,E_j,E_r,E_s)
\int_{Y\in\Sigma,~|Y|=1}Y^i\overline{Y^j}Y^r\overline{Y^s}d\theta(Y)\\
=&\frac{\text{V}(\mathbb{S}^{2k-1})}{k(k+1)}
\sum_{i,j=1}^k\Big[g(E_i,E_i,E_j,E_j)
+g(E_i,E_j,E_j,E_i)\Big].\qquad\big(\text{by (\ref{integral over sphere2})}\big)
\end{split}\nonumber\ee
\end{proof}

Applying Lemma \ref{integral formula lemma} to (\ref{RC-positivity}), (\ref{def riccik}) and (\ref{def sk}) produces the following alternative definitions as integrals over the unit sphere in $\Sigma$.
\begin{corollary}\label{relations riccik and sk}
Let the notation be as above. We have
\be\label{def integral uniform}R^{(E,h)}(\Sigma;u,\overline{u})=
\frac{k}{\text{V}(\mathbb{S}^{2k-1})}
\int_{Y\in\Sigma,~|Y|=1}R^{(E,h)}(Y,\overline{Y},u,\overline{u})d\theta(Y),\ee
\be\label{def integral riccik}\text{Ric}_k(x,\Sigma)(X,\overline{X})=
\frac{k}{\text{V}(\mathbb{S}^{2k-1})}\int_{Y\in\Sigma,~|Y|=1}R(X,\overline{X},Y,\overline{Y})d\theta(Y),\ee
and
\be\label{def integral sk} S_k(x,\Sigma)=\frac{k(k+1)}{2\text{V}(\mathbb{S}^{2k-1})}\int_{Y\in\Sigma,~|Y|=1}H(Y)d\theta(Y),\qquad\text{when $\omega$ is CKL}.\ee
\end{corollary}
\begin{remark}\label{remark2}
In \cite{NZ} the identity (\ref{def integral sk}) was taken as the definition of $S_k$ for K\"{a}hler metrics. Either formula of $\text{Ric}_k$ or $S_k$ has its own advantage. For instance, from (\ref{def riccik}) and (\ref{def sk}) the condition $\text{Ric}_k(x)>0$ (resp. $\text{Ric}_k>0$) implies $S_k(x)>0$ (resp. $S_k>0$). On the other hand, (\ref{def integral sk}) tells us that $S_k>0$ (resp. $S_k<0$) implies $S_{k+1}>0$ (resp. $S_{k+1}<0$). In contrast to it, in general the sign of $\text{Ric}_k$ is independent from that of $\text{Ric}_l$ when $k\neq l$, as illustrated by Hitchin's examples (\ref{Hitchin example}).
\end{remark}

\section{Proof of Theorem \ref{RC vanishing theorem}}\label{proof of rc vanishing}
\subsection{Two lemmas}
We prepare two crucial lemmas in order to establish Theorem \ref{RC vanishing theorem}. The following result was obtained in \cite[Lemma 2.1]{Li21} by adopting some arguments in \cite{NZ} and \cite{Yang18} , whose proof is to apply a $\partial\bar{\partial}$-Bochner formula and the maximum principle to part of directions.
\begin{lemma}\label{Bochner fromula0 lemma}
Let $u\in\Gamma(E)$ be a holomorphic section of the Hermitian holomorphic vector bundle $(E,h)$ over a compact complex manifold $M$, and the maximum of $|u|:=<u,u>^{\frac12}$ is attained at $x\in M$. Then
\be\label{Bochner formula0}<R_{X\overline{X}}(u),u>\big|_x\geq0,\qquad\text{for all $X\in T_xM.$}\ee
\end{lemma}
\begin{remark}
The use of a $\partial\bar{\partial}$-Bochner formula, together with the maximum principle to part of directions, was recently revived by some works (\cite{An}, \cite{AC}, \cite{Liu}, \cite{Ni13}, \cite{NZ}, \cite{Yang18}).
\end{remark}

The next lemma is parallel to \cite[Prop. 2.9]{Yang20}.
\begin{lemma}\label{equivalent relation}
Let $(E,h)$ be a uniformly RC $k$-positive Hermitian holomorphic vector bundle over a compact Hermitian manifold $(M,\omega)$. Then there exists a constant $C=C(h,\omega)>0$ such that for any $x\in M$, there exists a $k$-dimensional subspace $\Sigma_x\subset T_xM$ such that \be\label{inequality}R^{(E,h)}(\Sigma_x;u,\overline{u})\geq C,\qquad\text{for any unit vector $u\in E_x$}.\ee
\end{lemma}
\begin{proof}
For $x\in M$, let
$$\mathbb{G}_k(T_xM):=\Big\{\Sigma~\big|~\text{$\Sigma$ are  $k$-dimensional subspaces in $T_xM$}\Big\}$$
be the associated complex Grassmannian of $T_xM$ and $$\mathbb{S}E_x:=\big\{u\in E_x,~|u|=1\big\}$$ the unit sphere of $E_x$. Let
\be\label{def C(h,w)}C=C(h,\omega):=\min_{x\in M}\max_{\Sigma\in\mathbb{G}_k(T_xM)}
\min_{u\in\mathbb{S}E_x}R^{(E,h)}(\Sigma;u,\overline{u}),\ee
which is well-defined due to the compactness of $M$, $\mathbb{G}_k(T_xM)$ and $\mathbb{S}E_x$.

First note that such defined $C$ satisfies $(\ref{inequality})$. In fact, Definition (\ref{def C(h,w)}) implies that for any $x\in M$, $$\max_{\Sigma\in\mathbb{G}_k(T_xM)}
\min_{u\in\mathbb{S}E_x}R^{(E,h)}(\Sigma;u,\overline{u})\geq C.$$ This in turn yields that there exists a $\Sigma_x\in\mathbb{G}_k(T_xM)$ such that $$\min_{u\in\mathbb{S}E_x}R^{(E,h)}(\Sigma_x;u,\overline{u})\geq C,$$ which leads to (\ref{inequality}).

It suffices to show $C>0$. Suppose on the contrary that $C\leq0$. Then (\ref{def C(h,w)}) implies that there exists an $x_0\in M$ such that
$$\max_{\Sigma\in\mathbb{G}_k(T_{x_0}M)}
\min_{u\in\mathbb{S}E_{x_0}}R^{(E,h)}(\Sigma;u,\overline{u})=C\leq0.$$
This leads to $$\min_{u\in\mathbb{S}E_{x_0}}R^{(E,h)}(\Sigma;u,\overline{u})\leq0$$ for \emph{any} $\Sigma\in \mathbb{G}_k(T_{x_0}M)$, which in turn yields that for each $\Sigma\in\mathbb{G}_k(T_{x_0}M)$ there exists a $u(\Sigma)\in\mathbb{S}E_{x_0}$ such that $R^{(E,h)}\big(\Sigma;u(\Sigma),\overline{u(\Sigma)}\big)\leq0$.

In summary, under the assumption of $C\leq0$, we derive a conclusion that there exists an $x_0\in M$ such that, for every $\Sigma\in\mathbb{G}_k(T_{x_0}M)$, there exists $u(\Sigma)\in\mathbb{S}E_{x_0}$ such that $R^{(E,h)}\big(\Sigma;u(\Sigma),\overline{u(\Sigma)}\big)\leq0$, which exactly contradicts to the condition of uniform RC $k$-positivity of $(E,h)\rightarrow(M,\omega)$.
\end{proof}

\subsection{Proof of Theorem \ref{RC vanishing theorem}}
Let $T\in H^0\big(M,E^{\otimes p}\otimes (E^{\ast})^{\otimes q}\otimes F^{\otimes m}\big)$, $x\in M$ and a \emph{unit} vector $X\in T_xM$. Arbitrarily choose a Hermitian metric $h_F$ on $F$ and simply write $R$ to denote the Chern curvature tensors of various Hermitian holomorphic vector bundles involved in the proof.

Due to Lemma \ref{diagonal and real lemma} denote by the real eigenvalues of the Hermitian transformations $R_{X\overline{X}}(\cdot)$ on $\big(E_x,h(x)\big)$ and $\big(F_x,h_F(x)\big)$ by $\lambda_i=\lambda_i(x,X)$ ($1\leq i\leq r_1$) and $\mu_j=\mu_j(x,X)$ ($1\leq j\leq r_2$) respectively, where $r_1$ and $r_2$ are the ranks of $E$ and $F$. Let $\{e_1,\ldots,e_{r_1}\}$ and $\{s_1,\ldots,s_{r_2}\}$ be unitary bases of $E_x$ and $F_x$ such that
\begin{eqnarray}\label{0.1}
\left\{ \begin{array}{ll}
\displaystyle R_{X\overline{X}}(e_i)=\lambda_ie_i, &(1\leq i\leq r_1)\\
~\\
R_{X\overline{X}}(s_j)=\mu_js_j. &(1\leq j\leq r_2)
\end{array} \right.
\end{eqnarray}

Let $\{\theta^1,\ldots,\theta^{r_1}\}$ be the unitary basis of $E^{\ast}_x$ dual to $\{e_i\}$. Then the induced action of $R_{X\overline{X}}(\cdot)$ on $\{\theta^i\}$ is given by
\be\label{0.2}R_{X\overline{X}}(\theta^i)=-\lambda_i\theta^i,\qquad 1\leq i\leq r_1.\ee

Write
$$T\Big|_x=T^{\alpha_1\cdots\alpha_p\gamma_1\cdots\gamma_m}_{\beta_1\cdots\beta_q}e_{\alpha_1}
\otimes\cdots\otimes e_{\alpha_p}\otimes\theta^{\beta_1}\otimes\cdots\otimes\theta^{\beta_q}\otimes s_{\gamma_1}\otimes\cdots\otimes s_{\gamma_m}.$$
Then (\ref{0.1}) and (\ref{0.2}) imply
\be
\begin{split}R_{X\overline{X}}(T)\Big|_x
=\sum_{\begin{subarray}{1}
\alpha_1,\ldots,\alpha_p\\ \beta_1,\ldots,\beta_q\\ \gamma_1,\ldots,\gamma_m\end{subarray}}
\bigg[&\Big(\sum_{i=1}^p\lambda_{\alpha_i}-\sum_{j=1}^q
\lambda_{\beta_j}+\sum_{l=1}^m\mu_{\gamma_l}\Big)
T^{\alpha_1\cdots\alpha_p\gamma_1\cdots\gamma_m}_{\beta_1\cdots\beta_q}\\
&e_{\alpha_1}
\otimes\cdots\otimes e_{\alpha_p}\otimes
\theta^{\beta_1}\otimes\cdots\otimes\theta^{\beta_q}\otimes s_{\gamma_1}\otimes\cdots\otimes s_{\gamma_m}\bigg],\end{split}\nonumber\ee
which yields that
\be\label{T-formula}
\begin{split}&<R_{X\overline{X}}(T)\Big|_x,T\Big|_x>\\
=&
\sum_{\begin{subarray}{1}
\alpha_1,\ldots,\alpha_p\\ \beta_1,\ldots,\beta_q\\ \gamma_1,\ldots,\gamma_m\end{subarray}}
\bigg[\Big|T^{\alpha_1\cdots\alpha_p\gamma_1\cdots\gamma_m}_{\beta_1\cdots\beta_q}\Big|^2
\Big(\sum_{i=1}^p\lambda_{\alpha_i}-\sum_{j=1}^q\lambda_{\beta_j}+\sum_{l=1}^m\mu_{\gamma_l}\Big)\bigg].
\end{split}\ee

Let
\begin{eqnarray}\label{maxmin}
\left\{ \begin{array}{ll}
\lambda_{\max}(x,X):=\max\limits_{1\leq i\leq r_1}\big\{\lambda_i(x,X)\big\}\\
~\\
\lambda_{\min}(x,X):=\min\limits_{1\leq i\leq r_1}\big\{\lambda_i(x,X)\big\}\\
~\\
\lambda_{\max}:=\max\limits_{x\in M, X\in\mathbb{S}T_xM}\lambda_{\max}(x,X)\\
~\\
\lambda_{\min}:=\min\limits_{x\in M, X\in\mathbb{S}T_xM}\lambda_{\min}(x,X)
\end{array} \right.
\end{eqnarray}
and similarly define $\mu_{\max}(x,X)$, $\mu_{\min}(x,X)$, $\mu_{\min}$ and $\mu_{\max}$.
We remark that in general $\lambda_{\max}(x,X)$ and $\lambda_{\min}(x,X)$ are only \emph{continuous} functions and may not be smooth. Nevertheless, continuity is enough to guarantee that $\lambda_{\max}$ and $\lambda_{\min}$ are both well-defined real numbers as the maximum and minimum in (\ref{maxmin}) are over the unit sphere bundle of $TM$, which is compact. Clearly $\lambda_{\max}$ and $\lambda_{\min}$ depends only on $h$ and $\omega$. Moreover, continuity is also enough to do integration as we shall see in the following lemma. Accordingly, for various $\mu$'s similar properties hold.

An efficient upper bound estimate for (\ref{T-formula}) is exhibited by the next result.
\begin{lemma}\label{estimate of RXXTT}
Let $(E,h)\rightarrow(M,\omega)$ be a uniformly RC $k$-positive Hermitian holomorphic vector bundle. Let $C(h,\omega)$ be the positive constant and $\Sigma_x$ the desired $k$-dimensional subspace in $T_xM$ in Lemma \ref{equivalent relation}. Then for any $x\in M$, we have
\be\begin{split}\label{estimate formula}&\int_{X\in\Sigma_x,|X|=1}<R_{X\overline{X}}(T)\Big|_x,T\Big|_x>d\theta(X)\\
\leq&
\Big[\lambda_{\max}\cdot p-\frac{C(h,\omega)}{k}\cdot q+\mu_{\max}\cdot m\Big]\cdot\text{V}(\mathbb{S}^{2k-1})\cdot\big|T\big|^2\Big|_x.\end{split}\ee
\end{lemma}
\begin{proof}
To simplify the notation, denote by $\oint f(X)$ the integral of the function $f(X)$ over $\{X\in\Sigma_x~|~|X|=1\}\cong\mathbb{S}^{2k-1}$ with respect to the spherical measure $d\theta(X)$.

First we claim that
\be\label{0.5}\oint\mu_i(x,X)\leq
\mu_{\max}\cdot\text{V}(\mathbb{S}^{2k-1})\ee
and
\be\label{1}0<\frac{C(h,\omega)}{k}\cdot\text{V}(\mathbb{S}^{2k-1})
\leq\oint\lambda_i(x,X)\leq\lambda_{\max}\cdot\text{V}(\mathbb{S}^{2k-1}).\ee
In fact, (\ref{0.5}) and the last inequality in (\ref{1}) are obvious as $\mu_i(x,X)\leq\mu_{\max}$ and $\lambda_i(x,X)\leq\lambda_{\max}$. On the other hand, by definition $$\lambda_i(x,X)\overset{(\ref{0.1})}{=}<R_{X\overline{X}}(e_i),e_i>=R(X,\overline{X},e_i,\overline{e_i}).$$ Thus

\be\begin{split}
\oint\lambda_i(x,X)&=\oint R(X,\overline{X},e_i,\overline{e_i})\\
&=\frac{\text{V}(\mathbb{S}^{2k-1})}{k}R(\Sigma_x;e_i,\overline{e_i})
\qquad\big(\text{by (\ref{def integral uniform})}\big)\\
&\geq\frac{\text{V}(\mathbb{S}^{2k-1})}{k}\cdot C(h,\omega),
\qquad\big(\text{by (\ref{inequality})}\big)
\end{split}\nonumber\ee
which gives the desired one in (\ref{1}).

Taking integrals $\oint(\cdot)$ on both sides of (\ref{T-formula}) yields
\be\begin{split}&\oint<R_{X\overline{X}}(T)\Big|_x,T\Big|_x>\\
=&\sum_{\begin{subarray}{1}
\alpha_1,\ldots,\alpha_p\\ \beta_1,\ldots,\beta_q\\ \gamma_1,\ldots,\gamma_m\end{subarray}}
\Big|T^{\alpha_1\cdots\alpha_p\gamma_1\cdots\gamma_m}_{\beta_1\cdots\beta_q}\Big|^2
\oint\Big(\sum_{i=1}^p\lambda_{\alpha_i}-\sum_{j=1}^q\lambda_{\beta_j}+\sum_{l=1}^m\mu_{\gamma_l}\Big)\\
\leq&\sum_{\begin{subarray}{1}
\alpha_1,\ldots,\alpha_p\\ \beta_1,\ldots,\beta_q\\ \gamma_1,\ldots,\gamma_m\end{subarray}}
\Big|T^{\alpha_1\cdots\alpha_p\gamma_1\cdots\gamma_m}_{\beta_1\cdots\beta_q}\Big|^2
\Big[\lambda_{\max}\cdot p-
\frac{C(h,\omega)}{k}\cdot q+\mu_{\max}\cdot m\Big]\cdot\text{V}(\mathbb{S}^{2k-1})
\quad\big(\text{by (\ref{0.5}) and (\ref{1})}\big)\\
=&\Big[\lambda_{\max}\cdot p-\frac{C(h,\omega)}{k}\cdot q+\mu_{\max}\cdot m\Big]\cdot\text{V}(\mathbb{S}^{2k-1})\cdot\big|T\big|^2\Big|_x.
\end{split}\nonumber\ee
This gives the desired estimate (\ref{estimate formula}).
\end{proof}

We are ready now to complete the proof of Theorem \ref{RC vanishing theorem} in the following lemma.
\begin{lemma}
Let $(E,h)\rightarrow(M,\omega)$ be a uniformly RC $k$-positive Hermitian holomorphic vector bundle and $F$ any holomorphic vector bundle over $M$. With the notation above understood and set two positive constants by
\begin{eqnarray}\label{def of c1 c2}
C_1=C_1(h,\omega):=\displaystyle\frac{k\cdot\lambda_{\max}}{C(h,\omega)},\quad C_2=C_2(F,h,\omega):=\left\{ \begin{array}{ll}
\displaystyle\frac{k\cdot\mu_{\max}}{C(h,\omega)},&\text{if $\mu_{\max}>0$}\\
~\\
\displaystyle 1.&\text{if $\mu_{\max}\leq0$}
\end{array} \right.
\end{eqnarray}
Then for all $p,q,m\in\mathbb{Z}_{\geq 0}$ with $q> C_1\cdot p+C_2\cdot m$, we have
\be H^0\big(M,E^{\otimes p}\otimes (E^{\ast})^{\otimes q}\otimes F^{\otimes m}\big)=0\nonumber\ee
and hence Theorem \ref{RC vanishing theorem} holds.
\end{lemma}
\begin{proof}
First note that the positivity of $\lambda_{\max}$ and hence $C_1$ follows from the inequality (\ref{1}).
Let $T\in H^0\big(M,E^{\otimes p}\otimes (E^{\ast})^{\otimes q}\otimes F^{\otimes m}\big)$ and $\big|T\big|$ attains its \emph{maximum} at $x_0\in M$. On the one hand, Lemma \ref{Bochner fromula0 lemma} yields
\be\label{2}\int_{X\in \Sigma_{x_0},~|X|=1}<R_{X\overline{X}}(T)\Big|_{x_0},T\Big|_{x_0}>
d\theta(X)\geq0.\ee

On the other hand, applying Lemma \ref{estimate of RXXTT} to $x_0$ leads to
\be\label{3}\begin{split}&\int_{X\in \Sigma_{x_0},~|X|=1}<R_{X\overline{X}}(T)\Big|_{x_0},T\Big|_{x_0}>d\theta(X)\\
\leq&-\frac{C(h,\omega)}{k}\cdot
\Big[q-\big(\frac{k\lambda_{\max}}{C(h,\omega)}\cdot p+\frac{k\mu_{\max}}{C(h,\omega)}\cdot m\big)\Big]\cdot\text{V}(\mathbb{S}^{2k-1})\cdot\big|T\big|^2\Big|_{x_0}\\
\leq&
-\frac{C(h,\omega)}{k}\cdot
\Big[q-\big(C_1\cdot p+C_2\cdot m\big)\Big]\cdot
\text{V}(\mathbb{S}^{2k-1})\cdot\big|T\big|^2\Big|_{x_0}.
\quad\big(\text{by (\ref{def of c1 c2})}\big)
\end{split}\ee

If $q> C_1\cdot p+C_2\cdot m$, (\ref{2}) and (\ref{3}) together imply
$$\int_{X\in \Sigma_{x_0},~|X|=1}<R_{X\overline{X}}(T)\Big|_{x_0},T\Big|_{x_0}>d\theta(X)=0,$$
which, in turn via (\ref{3}) tells us that the only possibility is $\big|T\big|^2\Big|_{x_0}=0$. The maximum of $\big|T\big|$ at $x_0$ then implies $T\equiv0$. This completes the proof of this lemma and hence that of Theorem \ref{RC vanishing theorem}.
\end{proof}

\section{Proof of Theorem \ref{rick lead to uniform}}\label{proof of theorem rick yield rc}
\subsection{A lemma and the proof of Theorem \ref{rick lead to uniform}}
The proof of Theorem \ref{rick lead to uniform} is based on the
next result, which was essentially obtained by Ni-Zheng in \cite[Prop. 3.1]{NZ}, although the conclusion there was stated for K\"{a}hler manifolds.
\begin{lemma}[Ni-Zheng]\label{NZ estimate lemma}
Let $(M,\omega)$ be a compact Chern-K\"{a}hler-like Hermitian manifold, $x\in M$, and the $k$-dimensional subspace $\Sigma\subset T_xM$ minimize (resp. maximize) the $k$-scalar curvature $S_k(x,\cdot)$ at $x$. Then for any $Y\in\Sigma$ and $Z\in\Sigma^{\bot}:=\Big\{W\in T_xM~\big|~W\bot\Sigma\Big\}$, we have
\be\label{NZ estimate1}
\int_{X\in\Sigma,~|X|=1}R(X,\overline{X},Y,\overline{Z})d\theta(X)=
\int_{X\in\Sigma,~|X|=1}R(X,\overline{X},Z,\overline{Y})d\theta(X)=0,\ee
and
\be\label{NZ estimate2}
\int_{X\in\Sigma,~|X|=1}R(X,\overline{X},Z,\overline{Z})d\theta(X)\underset{(\leq)}
{\geq}\frac{\text{V}(\mathbb{S}^{2k-1})}{k(k+1)}\cdot S_k(x,\Sigma)\cdot|Z|^2.\ee
\end{lemma}
\begin{remark}
Lemma \ref{NZ estimate lemma} can be viewed as a $k$-dimensional generalization of \cite[Lemma 6.1]{Yang18} for holomorphic sectional curvature in the K\"{a}hler situation, which still holds true for CKL Hermitian metrics (\cite[Lemma 2.4]{Li21}). The proof in \cite[Prop. 3.1]{NZ}, in spirit of the similar principle to the one in \cite[Lemma 6.1]{Yang18}, is skillful. For the reader's convenience as well as for completeness, we include a detailed proof below. We shall see in the process of the proof that what we really need is various K\"{a}hler-type symmetries of the Chern curvature tensor, which, as explained in Remark \ref{remark}, is satisfied by the CKL Hermitian metrics.
\end{remark}

We first explain how Lemma \ref{NZ estimate lemma}, together with the materials in Section \ref{Preliminaries}, leads to Theorem \ref{rick lead to uniform} and postpone its proof to the next subsection.

\emph{Proof of Theorem \ref{rick lead to uniform}}.

\begin{proof}
Assume that $(M,\omega)$ is an $n$-dimensional compact \emph{Chern-K\"{a}hler-like} Hermitian manifold and $\text{Ric}_{k}(\omega)>0$ (resp. $\text{Ric}_{k}(\omega)<0$) for some $k\in\{1,\ldots,n\}$. Let
\be\label{lower bound for riccik} D:=\min_{\begin{subarray}{1}x\in M\\
\Sigma\in\mathbb{G}_k(T_xM)\\
X\in\Sigma,~|X|=1\end{subarray}}
\text{Ric}_k(x,\Sigma)(X,\overline{X})>0\qquad(\text{resp. $D:=\max\cdots<0$}).\ee
The definitions of $\text{Ric}_k(\omega)$ and $S_k(\omega)$ in (\ref{def riccik}) and (\ref{def sk}) then imply that
\be\label{lower bound for sk}\min_{\overset{x\in M}{\Sigma\in\mathbb{G}_k(T_xM)}}
S_k(x,\Sigma)\geq kD>0\qquad(\text{resp. $\max\cdots\leq kD<0$}).\ee

For any $x\in M$, let the $k$-dimensional subspace $\Sigma_x\subset T_xM$ \emph{minimize} (resp. \emph{maximize}) the $k$-scalar curvature $S_k(x,\cdot)$ at $x$, as required in Lemma \ref{NZ estimate lemma}. We shall show that such $\Sigma_x$ satisfies the inequality (\ref{RC-positivity}) and hence $(TM,\omega)$ is uniformly RC $k$-positive (resp. uniformly RC $k$-negative) over $(M,\omega)$. In fact, for any $X\in T_xM$, decompose $X$ as $X=X_1+X_2$, where $X_1\in\Sigma_x$ and $X_2\in\Sigma^{\bot}_x$. Like before, denote by $\oint f(Y)$ the integral of $f(Y)$ over $\{Y\in\Sigma_x~|~|Y|=1\}$ with respect to the spherical measure. Then
\be
\begin{split}
&R^{(TM,\omega)}(\Sigma_x;X,\overline{X})\\
=&\frac{k}{\text{V}(\mathbb{S}^{2k-1})}
\oint R(Y,\overline{Y},X,\overline{X})
\qquad\big(\text{by (\ref{def integral uniform})}\big)\\
=&\frac{k}{\text{V}(\mathbb{S}^{2k-1})}
\oint
R\big(Y,\overline{Y},X_1+X_2,
\overline{X_1}+\overline{X_2}\big)\\
=&\frac{k}{\text{V}(\mathbb{S}^{2k-1})}
\oint\Big[R\big(Y,\overline{Y},X_1,
\overline{X_1}\big)+
R\big(Y,\overline{Y},X_2,
\overline{X_2}\big)\Big]
\qquad\big(\text{by (\ref{NZ estimate1})}\big)\\
=&\text{Ric}_k(x,\Sigma)(X_1,\overline{X_1})+
\frac{k}{\text{V}(\mathbb{S}^{2k-1})}
\oint R\big(Y,\overline{Y},X_2,
\overline{X_2}\big)
\qquad\big(\text{by (\ref{def integral riccik})}\big)\\
\underset{(\leq)}{\geq}&\text{Ric}_k(x,\Sigma)(X_1,\overline{X_1})+
\frac{S_k(x,\Sigma)}{k+1}|X_2|^2\qquad\big(\text{by (\ref{NZ estimate2})}\big)\\
\underset{(\leq)}{\geq}&D|X_1|^2+\frac{kD}{k+1}|X_2|^2\qquad
\big(\text{by (\ref{lower bound for riccik}) and (\ref{lower bound for sk})}\big)\\
\underset{(\leq)}{\geq}&\frac{kD}{k+1}|X|^2.
\end{split}
\ee
This completes the proof of Theorem \ref{rick lead to uniform}.
\end{proof}

\subsection{Proof of Lemma \ref{NZ estimate lemma}}
The proof here basically follows the strategy in \cite[Prop. 3.1]{NZ}, but with more clear presentation in places.

Let $U(n)$ be the isometry group of $(T_xM,\omega_x=<,>)$ and
$$\mathfrak{u}(n):=
\Big\{a\in\text{Hom}(T_xM)\big|<a(W_1),W_2>+<W_1,a(W_2)>=0,~\forall~ W_1,W_2\in T_xM\Big\}$$
the (real) Lie algebra of $U(n)$. It is well-known that
$$e^{ta}:=\exp(ta):=\text{id}+\sum_{i=1}^{\infty}\frac{(ta)^i}{i!}\in U(n)\quad\text{and}\quad \frac{\text{d}^i}{\text{d}t^i}(e^{ta})=a^ie^{ta},\quad\text{
for any $t\in\mathbb{R}$}.$$

As before denote by $\oint f(X)d\theta(X)$ the integral of $f$ over $\{X\in\Sigma,|X|=1\}$, where $\Sigma$ is the $k$-dimensional subspace in $T_xM$ \emph{minimizing} (resp. \emph{maximizing}) $S_k(x,\cdot)$ at $x$.

For any $a\in\mathfrak{u}(n)$ consider the following \emph{real}-valued function
 $$f(t):=\oint H(e^{ta}X)d\theta(X),\qquad t\in\mathbb{R}.$$
The minimum (resp. maximum) of $f(t)$ at $t=0$ implies that $f'(0)=0$ and $f''(0)\geq0$ (resp. $\leq0$).
Direct calculations, together with the fact that the Chern curvature tensor $R$ satisfy K\"{a}hler-like symmetries ensured by the condition of CKL, lead to
\be\label{5}\oint\Big[R\big(a(X),\overline{X},X,\overline{X}\big)+
R\big(X,\overline{a(X)},X,\overline{X}\big)\Big]d\theta(X)=0\ee
and
\be\label{6}\begin{split}
\oint\Big[&R\big(a^2(X),\overline{X},X,\overline{X}\big)+
R\big(X,\overline{a^2(X)},X,\overline{X}\big)+
4R\big(a(X),\overline{a(X)},X,\overline{X}\big)\\
+&R\big(a(X),\overline{X},a(X),\overline{X}\big)+
R\big(X,\overline{a(X)},X,\overline{a(X)}\big)
\Big]d\theta(X)\underset{(\leq)}{\geq}0.
\end{split}\ee

We may assume that the two vectors $Y\in\Sigma$ and $Z\in\Sigma^{\bot}$ are both \emph{nonzero} as otherwise (\ref{NZ estimate1}) and (\ref{NZ estimate2}) trivially hold. Let
$$a(\cdot):=\sqrt{-1}\big(<\cdot,Y>Z+<\cdot,Z>Y\big)$$
and it can be easily checked that $a(\cdot)\in\mathfrak{u}(n)$.
Then
\be\label{a formula}a(X)=\sqrt{-1}<X,Y>Z,\quad a^2(X)=-<X,Y>|Z|^2Y.\ee

Apply (\ref{5}) to (\ref{a formula}) and also the one with $Z$ being replaced by $\sqrt{-1}Z$, and sum the two we have
\be\label{8}\oint<X,Y>R(Z,\overline{X},X,\overline{X})d\theta(X)=0.\ee
We take a unitary basis of $\Sigma$ as follows \be\label{7}\big\{E_1=Y/|Y|,E_2,\ldots,E_k\big\},\ee
and simply write $V:=\text{V}(\mathbb{S}^{2k-1})$. Then
\be\begin{split}
0=&\oint<X,Y>R(Z,\overline{X},X,\overline{X})d\theta(X)\qquad\big(\text{by (\ref{8})}\big)\\
=&\frac{V}{k(k+1)}
\sum_{i,j=1}^k\Big[<E_i,Y>R(Z,\overline{E_i},E_j,\overline{E_j})+
<E_i,Y>R(Z,\overline{E_j},E_j,\overline{E_i})\Big]~\big(\text{by (\ref{integral formula2})}\big)\\
=&\frac{2V}{k(k+1)}\sum_{j=1}^k
R(Z,\overline{Y},E_j,\overline{E_j})\quad\big(\text{by (\ref{7}) and $\omega$ CKL}\big)\\
=&\frac{2}{k+1}\oint R(Z,\overline{Y},X,\overline{X})d\theta(X).
\quad\big(\text{by (\ref{integral formula1})}\big)
\end{split}
\ee
This and its conjugate lead to (\ref{NZ estimate1}).

We now derive (\ref{NZ estimate2}). Apply (\ref{6}) to the above (\ref{a formula}) and also the one being replaced by $\sqrt{-1}Z$, and sum the two we have
\be\label{9}\begin{split}
&4\oint|<X,Y>|^2R(Z,\overline{Z},X,\overline{X})d\theta(X)\\
\underset{(\leq)}{\geq}&
\oint\Big[<X,Y>R(Y,\overline{X},X,\overline{X})+
<Y,X>R(X,\overline{Y},X,\overline{X})\Big]|Z|^2d\theta(X).
\end{split}\ee
Arbitrarily choose a unitary basis $\{E_1,\ldots,E_k\}$ of $\Sigma$, which is irrelevant to (\ref{7}). Then
\be\label{10}\begin{split}
&\oint\Big[\text{LHS of (\ref{9})}\Big]d\theta(Y)\\
=&\frac{4V}{k}\sum_{i=1}^k\oint|<X,E_i>|^2R(Z,\overline{Z},X,\overline{X})d\theta(X)
\qquad\big(\text{by (\ref{integral formula1})}\big)\\
=&\frac{4V^2}{k^2(k+1)}\sum_{i,j,t=1}^k
\Big[|<E_j,E_i>|^2R(Z,\overline{Z},E_t,\overline{E_t})\\ &+<E_j,E_i>\overline{<E_t,E_i>}R(Z,\overline{Z},E_t,\overline{E_j})\Big]
\qquad\big(\text{by (\ref{integral formula2})}\big)\\
=&\frac{4V^2}{k^2}\sum_{i=1}^kR(Z,\overline{Z},E_i,\overline{E_i})\\
=&\frac{4V}{k}\oint R(Z,\overline{Z},X,\overline{X})d\theta(X).
\qquad\big(\text{by (\ref{integral formula1})}\big)
\end{split}\ee
Analogous to (\ref{10}) it can be derived that
\be\label{11}
\oint\Big[\text{RHS of (\ref{9})}\Big]d\theta(Y)=
\frac{4V^2}{k^2(k+1)}S_k(x,\Sigma)|Z|^2.
\ee
Putting (\ref{9}), (\ref{10}) and (\ref{11}) together yields the desired
(\ref{NZ estimate2}).

\section{Further questions and remarks}\label{further remarks}
In the process of proofs we have seen that the \emph{strict} positivity or negativity in Theorems \ref{RC vanishing theorem} and \ref{rick lead to uniform} is needed.
In view of Theorem \ref{Bochner}, it is natural to
wonder whether the conclusions remain true if only assuming
quasi-positivity or quasi-negativity of various curvature conditions.
Therefore the following question can be proposed (see \cite[Conjecture 1.9]{Yang20}).
\begin{question}\label{question}
Whether the conclusions in Theorems \ref{RC vanishing theorem},
\ref{rick lead to uniform}, \ref{CKL result} and
\ref{rational connectedness} hold true if various positivity or negativity conditions
 are weakened respectively to quasi-positivity or quasi-negativity?
\end{question}
\begin{remark}
What Heier-Wong showed in \cite{HW} is indeed that
a projective K\"{a}hler manifold with quasi-positive
holomorphic sectional curvature is rationally connected, which provides some positive evidence towards Question \ref{question}. The condition of quasi-positivity in this case was further strengthened by Matsumura in \cite[Thm 1.2]{Ma}. By combining this with some arguments in \cite{Yang20} and \cite{HLWZ}, S. Zhang and X. Zhang showed that the quasi-positivity of $H=\text{Ric}_1$ for a compact K\"{a}hler manifold leads to projectivity and rational-connectedness (\cite{ZZ23}). For more recent results along this line, we refer the reader to \cite{Ta24}, \cite{Ni25} and \cite{DN25}.
\end{remark}

Among various geometric positivity concepts for Hermitian holomorphic vector bundles, the \emph{Griffiths positivity} (\cite{Gr})
may be the best-known. In view of the results proved in this article, it may be interesting to propose the
following notion.
\begin{definition}\label{Griffith kl definition}
A rank $r$ Hermitian holomorphic vector bundle $(E,h)$ over an $n$-dimensional Hermitian manifold $(M,\omega)$
is called \emph{Griffiths $(k,l)$-positive} $(1\leq k\leq n,~1\leq l\leq r)$ if at each point $x\in M$, and
for any $\Sigma\in\mathbb{G}_k(T_xM)$ and any
$\sigma\in\mathbb{G}_l(E_x)$, we have
\be\label{Griffith kl formula}
R^{(E,h)}(\Sigma;\sigma):=
\sum_{\underset{1\leq j\leq l}{1\leq i\leq k}}
R^{(E,h)}(E_i,\overline{E_i},e_j,\overline{e_j})>0,\ee
where $\{E_1,\ldots E_k\}$ (resp. $\{e_1,\ldots,e_l\}$) is a unitary basis of
$\Sigma$ (resp. $\sigma$). \emph{Griffiths $(k,l)$-negativity} and various quasi-versions can be similarly defined.
\end{definition}
\begin{remark}
\begin{enumerate}
\item
By (\ref{integral formula1}), an alternative definition for $R^{(E,h)}(\Sigma;\sigma)$ is
\be\begin{split}
&R^{(E,h)}(\Sigma;\sigma)\\
=&
\frac{k^2}{\text{V}(S^{2k-1})\text{V}(S^{2l-1})}
\iint_{\underset{\{X\in\Sigma,~|X|=1\}}{\{u\in\Gamma,~|u|=1\}}}
R^{(E,h)}(X,\overline{X},u,\bar{u})d\theta(X)d\theta(u).
\end{split}\nonumber\ee

\item
Griffiths $(1,1)$-positivity is the original Griffiths positivity, in which case the metric $\omega$ is irrelevant. By definition Griffiths $(k,l)$-positivity (resp. negativity) implies Griffiths $(k+1,l)$-positivity and $(k,l+1)$-positivity (resp. negativity). Thus the condition Grffiths $(k,l)$-positivity becomes weaker as $k$ or $l$ increases (compare to the $k$-scalar curvatures in Remark \ref{remark2}).

\item
The subject of cohomology vanishing theorems for Griffiths positive holomorphic vector bundles occupies a central role in several complex variables and algebraic geometry (\cite[\S 7]{De}, \cite[\S 6]{SS}). It seems to be possible to generalize these to Griffiths $(k,l)$-positive vector bundles.
\end{enumerate}
\end{remark}

Another two versions of $(k,l)$-positivity (resp. negativity) related to Definition \ref{uniform rc k positivity} are as follows.
\begin{definition}\label{uniform RC kl definition}
Let $(E,h)$ be a rank $r$ Hermitian holomorphic vector bundle over an $n$-dimensional Hermitian manifold $(M,\omega)$,
and $k\in\{1,\ldots,n\}$ and $l\in\{1,\ldots,r\}$.
\begin{enumerate}
\item
It is called \emph{RC $(k,l)$-positive} \big(resp. \emph{BC $(k,l)$-positive}\big) at $x\in M$ if for any $\sigma\in\mathbb{G}_l(E_x)$ \big(resp. $\Sigma\in\mathbb{G}_k(T_xM)$\big), there exists $\Sigma\in \mathbb{G}_k(T_xM)$ \big(resp. $\sigma\in\mathbb{G}_l(E_x)$\big) such that $R^{(E,h)}(\Sigma;\sigma)>0$. If this holds at each $x\in M$, then it is called RC $(k,l)$-positive \big(resp. BC $(k,l)$-positive\big).

\item
It is called \emph{uniformly RC $(k,l)$-positive} \big(resp. \emph{uniformly BC $(k,l)$-positive}\big) at $x\in M$ if there exists $\Sigma\in\mathbb{G}_k(T_xM)$ \big(resp. $\sigma\in\mathbb{G}_l(E_x)$\big) such that for every $\sigma\in\mathbb{G}_l(E_x)$ \big(resp. $\Sigma\in\mathbb{G}_k(T_xM)$\big), $R^{(E,h)}(\Sigma;\sigma)>0.$ If this holds at each $x\in M$, then it is called uniformly RC $(k,l)$-positive \big(resp. uniformly BC $(k,l)$-positive\big).
\end{enumerate}
(Uniform) RC $(k,l)$-negativity or BC $(k,l)$-negativity and various quasi-versions can be similarly defined.
\end{definition}
\begin{remark}\label{uniform RC remark}
\begin{enumerate}
\item
Uniform RC $(k,1)$-positivity is the one in Definition \ref{RC-positivity}, RC $(1,1)$-positivity is the RC-positivity in \cite{Yang18}, and BC $(1,l)$-positivity is the BC $l$-positivity in \cite[p. 280]{Ni21-2}. For CKL Hermitian manifold $(M,\omega)$, $(TM,\omega)$ is (uniformly) RC $(k,l)$-positive (resp. negative) if and only if it is (uniformly) BC $(l,k)$-positive (resp. negative).

\item
It turns out that BC $k$-positivity of the tangent bundle of a Hermitian metric for some $k$ implies $h^{k,0}=0$ (\cite[Coro. 4.4]{Ni21-2}). It may be interesting to explore possible consequences of geometric importance for general uniform RC (resp. BC) $(k,l)$-positivity.

\item
By the proof of Lemma \ref{equivalent relation}, it is easy to see that RC $(k,l)$-positivity or BC $(k,l)$-positivity on a \emph{compact} Hermitian manifold amounts to
$$\min_{x\in M}
\min_{\sigma\in\mathbb{G}_l(E_x)}
\max_{\Sigma\in\mathbb{G}_k(T_xM)}R^{(E,h)}(\Sigma;\sigma)>0$$
or
$$\min_{x\in M}\min_{\Sigma\in\mathbb{G}_k(T_xM)}
\max_{\sigma\in\mathbb{G}_l(E_x)}
R^{(E,h)}(\Sigma;\sigma)>0,$$
and uniform RC $(k,l)$-positivity or uniform BC $(k,l)$-positivity amounts to
$$\min_{x\in M}\max_{\Sigma\in\mathbb{G}_k(T_xM)}
\min_{\sigma\in\mathbb{G}_l(E_x)}R^{(E,h)}(\Sigma;\sigma)>0$$
or
$$\min_{x\in M}\max_{\sigma\in\mathbb{G}_l(E_x)}
\min_{\Sigma\in\mathbb{G}_k(T_xM)}R^{(E,h)}(\Sigma;\sigma)>0$$
respectively.

\item
Besides the curvatures mentioned in this article, some other interesting curvature notions were also introduced by Ni and Ni-Zheng. We refer the reader to their survey article \cite{NZ19} for more details.
\end{enumerate}
\end{remark}

\end{document}